\documentclass[oneside,english]{amsart}
\usepackage[T1]{fontenc}
\usepackage[latin9]{inputenc}
\usepackage{amsthm}
\usepackage{amstext}
\usepackage{amssymb}
\usepackage{esint}

\makeatletter
\numberwithin{equation}{section}
\numberwithin{figure}{section}
\theoremstyle{plain}
\newtheorem{thm}{\protect\theoremname}[section]

\theoremstyle{plain}
\theoremstyle{definition}
\newtheorem{defn}[thm]{\protect\definitionname}
\theoremstyle{plain}
\newtheorem{lem}[thm]{\protect\lemmaname}
\newtheorem{cor}[thm]{\protect\corollaryname}
\theoremstyle{plain}
\newtheorem{rem}[thm]{\protect\remarkname}
\theoremstyle{plain}
\newtheorem{exa}[thm]{\protect\examplename}
\theoremstyle{plain}
\makeatother

\usepackage{babel}
\providecommand{\definitionname}{Definition}
\providecommand{\lemmaname}{Lemma}
\providecommand{\theoremname}{Theorem}
\providecommand{\corollaryname}{Corollary}
\providecommand{\remarkname}{Remark}
\providecommand{\propositionname}{Proposition}
\providecommand{\examplename}{Example}
	
\DeclareMathOperator{\loc}{loc}

\DeclareMathOperator{\ess}{ess}
\DeclareMathOperator{\diam}{diam}
\DeclareMathOperator{\adj}{adj}
\DeclareMathOperator{\cp}{cap}

\begin{document}

\title[The spectral estimates for the Neumann-Laplace operator]
{The spectral estimates for the Neumann-Laplace operator in space domains}

\author{V.~Gol'dshtein, A.~Ukhlov}
\begin{abstract}

In this paper we prove discreteness of the spectrum of the Neu\-mann-Lap\-la\-ci\-an (the free membrane problem) in a large class of non-convex space domains. The  lower estimates of the first non-trivial eigenvalue are obtained in terms of geometric characteristics of Sobolev mappings. The suggested approach is based on Sobolev-Poincar\'e inequalities that are obtained with the help of a geometric theory of composition operators on Sobolev spaces. These composition operators  are induced by generalizations of conformal mappings that is called as mappings of bounded $2$-distortion (weak $2$-quasiconformal mappings).

\end{abstract}
\maketitle
\footnotetext{\textbf{Key words and phrases:} Sobolev spaces, elliptic equations, quasiconformal mappings} \footnotetext{\textbf{2000
Mathematics Subject Classification:} 46E35, 35P15, 30C65.}

\section{Introduction }

The classical upper estimate for the first nontrivial Neumann eigenvalue of the Laplace operator   
$$
\mu_1(\Omega)\leq \mu_1(\Omega^{\ast})=\frac{p^2_{n/2}}{R^2_{\ast}}
$$
was proved by Szeg\"o \cite{S54} for simply connected planar domains via a conformal mappings technique ("the method of conformal normalization") and by Weinberger  \cite{W56} for  domains in $\mathbb{R}^n$. In this inequality $p_{n/2}$ denotes the first positive zero of the function $(t^{1-n/2}J_{n/2}(t))'$, and $\Omega^{\ast}$ is an $n$-ball of the same $n$-volume as $\Omega$ with $R_{\ast}$ as its radius. In particular, if $n=2$ we have $p_1=j_{1,1}'\approx1.84118$ where $j_{1,1}'$ denotes the first positive zero of the derivative of the Bessel function $J_1$. 

More detailed upper estimates for planar domains  were obtained in \cite{PS51} and \cite{H65} via "the method of conformal normalization". The upper estimates of the Laplace eigenvalues with the help of different techniques were intensively studied in the recent decades, see, for example, \cite{ A98, AB93, AL97, EP15, LM98}. 

Situation with lower estimates is more complicated. The classical result by Payne and Weinberger  \cite{PW} states that in convex domains $\Omega\subset\mathbb R^n$, $n\geq 2$ 
$$
\mu_1(\Omega)\geq \frac{\pi^2}{d(\Omega)^2},
$$
where $d(\Omega)$ is a diameter of a convex domain $\Omega$. 
Unfortunately in non-convex domains $\mu_1(\Omega)$ can not be estimated in the terms of Euclidean diameters. It can be seen by considering a domain consisting of two identical squares connected by a thin corridor \cite{BCDL16}. In \cite{BCT9, BCT15} lower estimates involved the isoperimetric constant relative to $\Omega$ were obtained. 

In the works  \cite{GU15, GU15_2} we returned to a conformal mappings techniques and obtained lower estimates of $\mu_1(\Omega)$  in the terms of the hyperbolic (conformal) radius of $\Omega$ for a large class of general (non necessary convex) domains $\Omega\subset\mathbb R^2$. For example, this class includes some domains with fractal boundaries which Hausdorff dimension can be any number of the half interval $[1,2)$.

Our method is different from "the method of conformal normalization" and based on the variational formulation of spectral problems and on the geometric theory of composition operators  on Sobolev spaces, developed in our previous papers \cite{GGR,U1,UV10,VGR}. Roughly speaking we "transferred" known estimates (from convex Lipschitz domains) to "general" domains with a help of composition operators induced by conformal mappings.

The variational formulation of the spectral problem for the Laplace operator is usually based on the Dirichlet (energy) integral
$$
\|u \mid L^1_2(\Omega)\|^2=\int\limits_{\Omega}|\nabla u(x)|^2~dx,
$$ 
and was established in \cite{R94} by Lord Rayleigh.

In the present work we suggest lower estimates of the first nontrivial Neumann eigenvalues for a large class of non-convex spaces domains using a natural generalization of conformal mappings that we call weak $2$-quasiconformal mappings (topological mappings of bounded $2$-distortion) $\varphi:\Omega \to \widetilde{\Omega} $ that induce bounded composition operator $\varphi^{\ast}: L^1_2(\widetilde{\Omega}) \rightarrow  L^1_2(\Omega)$ \cite{GGR}. Here $\Omega$ is a Lipschitz domain (where discreteness of the spectrum is known) and $\widetilde{\Omega}$ is its image that can be a much more irregular domain. Our method is based on applications of the geometric theory of composition operators on Sobolev spaces to Sobolev-Poincar\'e inequalities \cite{GGu,GU}.

As a motivation for  "naturalness" of the class of weak $2$-quasiconformal homeomorphisms  let us check firstly how the energy integrals are changed under diffeomorphisms $\varphi: \Omega\to\widetilde{\Omega}$. By the chain rule:
$$
\int\limits_{\Omega}|\nabla u\circ\varphi(x)|^2~dx \leq
\int\limits_{\Omega}|\nabla u|^2(\varphi(x))\frac{|D\varphi(x)|^2}{|J(x,\varphi)|}|J(x,\varphi)|dx .
$$
If the point-wise dilatation 
$$
K(x,\varphi)=\frac{|D\varphi(x)|^2}{|J(x,\varphi)|}
$$ 
is bounded a.~e. in $\Omega$, then by the classical change of variable in the Lebesgue integral formula we obtain: 

\begin{multline}
\|u\circ\varphi \mid L^1_2(\Omega)\|=\left(\int\limits_{\Omega}|\nabla u\circ\varphi(x)|^2~dx\right)^{\frac{1}{2}}
\\
\leq\ess\sup\limits_{x\in\Omega}\left(\frac{|D\varphi(x)|^2}{|J(x,\varphi)|}\right)^{\frac{1}{2}}
\left(\int\limits_{\tilde{\Omega}}|\nabla u|^2~dy\right)^{\frac{1}{2}}=
\ess\sup\limits_{x\in\Omega}\left(K(x,\varphi)\right)^{\frac{1}{2}}\|u\mid L^1_2(\widetilde{\Omega})\|.
\nonumber
\end{multline}

Here the value $\ess\sup\limits_{x\in\Omega}\left(K(x,\varphi)\right)^{\frac{1}{2}}$ is an upper bound of the norm of the composition operator $\varphi^{\ast}$ on Sobolev spaces $L^1_2$ defined by the composition rule $\varphi^{\ast}(f)=f\circ\varphi$.

Let us call the quantity
$$
K_p(\Omega)=\left(\ess\sup\limits_{x\in\Omega}\frac{|D\varphi(x)|^p}{|J(x,\varphi)|}\right)^{\frac{1}{p}}, \,\,1\leq p<\infty,
$$
the $p$-dilatation of $\varphi$. Sobolev mappings with bounded $p$-dilatation were called weak $p$-quasi\-con\-for\-mal  \cite{GGR} or mappings of bounded $p$-distortion \cite{UV10}. Only these mappings generate (by the composition rule) bounded composition operators  on the Sobolev spaces $L^1_p$ 
$$
\varphi^{\ast}: L^1_p(\widetilde{\Omega})\to L^1_p(\Omega), \,\,1\leq p<\infty.
$$
 
We are focus here on the case $p=2$  for the study of spectral properties of the Laplace operator. 

Let us give few essential remarks:

1) In planar domains the equality
$$
\frac{|D\varphi(x)|^2}{J(x,\varphi)}=1, \,\,\text{for all}\,\,x\in\Omega,
$$
is one of the classical definitions of the conformal mappings. In this case
$$
\|u\circ\varphi \mid L^1_2(\Omega)\|=\|u\mid L^1_2{\widetilde{\Omega}}\|,
$$
i.~e. any conformal mapping induces an isometry of $L^1_2$-spaces. Unfortunately, for dimension more then two, isometries of $L^1_2$-spaces can be induced by isometries of $\mathbb{R}^n$ only. 

2) In planar domains boundedness of 
$$
K_2(\Omega)=\left(\ess\sup\limits_{x\in\Omega}\frac{|D\varphi(x)|^2}{|J(x,\varphi)|}\right)^{\frac{1}{2}}
$$
is one of classical definitions of quasiconformal mappings. 

3) The class of space mappings of bounded $2$-distortion is more flexible then the class of space quasiconformal homeomorphisms ($n$-quasiconformal in our notation). For example, there exists  a mapping of bounded $2$-distortion of the unit ball onto a ridge \cite{GGR}. 

4) In the space $\mathbb R^3$ boundedness of $K_2(\Omega)$ can be expressed as a boundedness of the ratio of the differential and co-differential and it permit us to call alternatively homeomorphisms of bounded $2$-distortion as {\it co-quasiconformal mappings}.

By our opinion the class of mappings of bounded $2$-distortion (weak co-quasi\-con\-for\-mal mappings) is very natural for study of spectral problems for elliptic operators in space domains. Using these mappings we prove solvability of the spectral problem in space domains with  H\"older singularities and we obtain lower estimates of the first non-trivial Neumann eigenvalue for Laplacian in non-convex space domains, in particular, in domains with H\"older singularities.

The proposed "transfer" procedure is based on the commutative diagram:
\[\begin{array}{rcl}
L^{1}_{2}(\widetilde{\Omega}) & \stackrel{\varphi^*}{\longrightarrow} & L^1_2(\Omega) \\[2mm]
\multicolumn{1}{c}{\downarrow} & & \multicolumn{1}{c}{\downarrow} \\[1mm]
L_2(\widetilde{\Omega}) & \stackrel{(\varphi^{-1})^*}{\longleftarrow} & L_2(\Omega).
\end{array}\]

Here horizontal arrows correspond to composition operators for Sobolev space $L^{1}_{2}$ and Lebesgue spaces $L_2$, vertical arrows correspond to Poincar\'e inequalities.

The main results will be discussed in the next section. Let us remark only that the "transfer"' procedure permit us to prove discreteness of the spectrum and give estimates for the first nontrivial  Neumann eigenvalues of the Laplace operator in domains with anisotropic H\"older singularities. 

The paper is organized as follows:
Basic definitions, few main results and some applications of main results are presented in section 2. In section 3 we discuss applications of composition operators to the Sobolev-Poincar\'e inequalities and apply them to lower estimates of the first nontrivial Neumann eigenvalue for Laplace operator. In section 4 we apply main results to domains with anisotropic 
H\"older singularities. Section 5 is devoted to an extension of main results to Neumann eigenvalues of $p$-Laplace operators. In Section 6 we discuss some necessary properties of mappings of bounded $p$-dilatation.

\section{The eigenvalue problem and main results}


Let  $\Omega\subset\mathbb R^n$ be a bounded Lipschitz domain (an open connected set).  The Neumann eigenvalue problem
for the Laplace operator is:
\begin{equation}
\label{Laplace}
\begin{cases}
-\operatorname{div}\left(\nabla u\right)=\mu u& \text{in}\,\,\,\,\,\Omega\\
 \frac{\partial u}{\partial n}=0& \text{on}\,\,\,\partial\Omega.
\end{cases}
\end{equation}

The weak statement of this spectral problem is as follows: a function
$u$ solves the previous problem iff $u\in W^{1}_{2}(\Omega)$ and 
\[
\int\limits _{\Omega} \nabla u(x)\cdot\nabla v(x)~dx=\mu \int\limits _{\Omega}u(x)v(x)~dx
\]
for all $v\in W^{1}_{2}(\Omega)$. This statement is correct in any bounded domain $\Omega\subset\mathbb R^n$.

Recall that the Sobolev space $W^1_p(\Omega)$, $1\leq p<\infty$, is defined
as a Banach space of locally integrable weakly differentiable functions
$f:\Omega\to\mathbb{R}$ equipped with the following norm: 
\[
\|f\mid W^1_p(\Omega)\|=\biggr(\int\limits _{\Omega}|f(x)|^{p}\, dx\biggr)^{\frac{1}{p}}+
\biggr(\int\limits _{\Omega}|\nabla f(x)|^{p}\, dx\biggr)^{\frac{1}{p}}.
\]

The homogeneous Sobolev space $L^1_p(\Omega)$, $1\leq p<\infty$, is defined as a seminormed space
of locally integrable weakly differentiable functions $f:\Omega\to\mathbb{R}$ equipped
with the following seminorm: 
\[
\|f\mid L^1_p(\Omega)\|=\biggr(\int\limits _{\Omega}|\nabla f(x)|^{p}\, dx\biggr)^{\frac{1}{p}}.
\]

\begin{rem} We suppose that any function $u$ from the spaces $W^1_p$ or $L^1_p$ is defined $p$-quasi-everywhere by its Lebesgue value (so-called refined or $p$-quasi-continuous functions) i.e. is defined everywhere outside of a set of $p$-capacity zero. For example, it means that for $p>n$ any function $u\in L^1_p$ is continuous. This result appeared first in \cite{MH}.
\end{rem}

By the Min-Max Principle (see, for example, \cite{CG}) the first nontrivial Neumann eigenvalue for the Laplacian can be characterized as
$$
\mu_1(\Omega)=\min\left\{\frac{\int\limits_{\Omega}|\nabla u(x)|^2~dx}{\int\limits_{\Omega}|u(x)|^2~dx}: u\in W^{1}_{2}(\Omega)\setminus\{0\},\,
\int\limits_{\Omega}u~dx=0\right\}.
$$

Hence $\mu_1(\Omega)^{-\frac{1}{2}}$ is the best constant $B_{2,2}(\Omega)$  in the following Poincar\'e inequality
$$
\inf\limits_{c\in\mathbb R}\|f-c \mid L_2(\Omega)\|\leq B_{2,2}(\Omega)\|\nabla f \mid L_2(\Omega)\|,\,\,\, f\in W^{1}_{2}(\Omega).
$$

The calculation of $\mu_1(\Omega)$ is possible in very limited cases. Hence, the problem of estimates of the first nontrivial Neumann eigenvalue for Laplacian is a significant problem in the modern analysis of PDE.

In the present work we suggest the lower estimates the Neumann eigenvalues in spaces domains via derivatives of weak $2$-quasiconformal mappings.
Using the geometric theory of composition operators on Sobolev spaces (i.e. the theory of mappings of bounded $p$-distortion) we prove the following result for 
a general class of domains:

\vskip 0.5cm

{\bf Theorem A.} {\it  
Suppose that there exists a weak $2$-quasiconformal homeomorphism  $\varphi: \Omega\to\widetilde{\Omega}$ of a bounded Lipschitz  domain 
$\Omega\subset\mathbb R^n$ onto $\widetilde{\Omega}$, such that 
$$
M_2(\Omega)=\operatorname{ess}\sup\limits_{x\in \Omega}\left|J(x,\varphi)\right|^{\frac{1}{2}}<\infty.
$$
Then the spectrum of Neumann-Laplace operator in $\widetilde{\Omega}$ is discrete, can be written in the form
of a non-decreasing sequence
\[
0=\mu_0(\widetilde{\Omega})<\mu_{1}(\widetilde{\Omega})\leq\mu_{2}(\widetilde{\Omega})\leq...\leq\mu_{n}(\widetilde{\Omega})\leq...\,,
\]
and
$$
\frac{1}{\mu_1(\widetilde{\Omega})}\leq K^2_{2}(\Omega) M^2_2(\Omega)\frac{1}{\mu_1(\Omega)}.
$$
}

\vskip 0.5cm

Here $K_{2}(\Omega)$ is the coefficient of a weak $2$-quasiconformality of the homeomorphism $\varphi:\Omega\to\widetilde{\Omega}$ and $J(x,\varphi)$ is the determinant of the Jacobi matrix of $\varphi$ at $x$.

Recall that a homeomorphism $\varphi:\Omega\to\widetilde{\Omega}$ is called weak $p$-quasiconformal \cite{GGR}, $1\leq p < \infty$, if $\varphi\in W^1_{1,\loc}(\Omega)$,
has finite distortion and the quantity (that we call the coefficient of weak $p$-quasiconformality) 
$$
K_p(\Omega)=\left(\ess\sup\limits_{x\in\Omega}\frac{|D\varphi(x)|^p}{|J(x,\varphi)|}\right)^{\frac{1}{p}}
$$
is finite. 

This Theorem~A is not strong enough for H\"older type singularities. The corresponding strong version (Theorem C) will be formulated and proved later. Here we demonstrate a corollary of Theorem C for comparatively simple case. Denote by $H_1$ the standard $n$-dimensional simplex, $n\geq 3$,
$$
H_1:=\{ x\in\mathbb R^n : n \geq 3, 0<x_n<1, 0<x_i<x_n,\,\,i=1,2,\dots,n-1\}.
$$

\vskip 0.5cm

{\bf Theorem B.} {\it  
Let 
$$
H_g:=\{ x\in\mathbb R^n : n \geq 3, 0<x_n<1, 0<x_i<x_n^{\gamma_i},
\,i=1,2,\dots,n-1\}
$$  
$\gamma_i \geq 1$, $\gamma:=1+\sum_{i=1}^{n-1}\gamma_i$ $, g:=(\gamma_1,...,\gamma_{n-1})$ .

Then the spectrum of the Neumann-Laplace operator in the domain $H_g$ is discrete, can be written in the form
of a non-decreasing sequence
\[
0=\mu_0(H_g)<\mu_{1}(H_g)\leq\mu_{2}(H_g)\leq...\leq\mu_{n}(H_g)\leq...\,,
\]
and for any $r>2$ the following inequality holds:
\begin{multline}
\frac{1}{\mu_1(H_g)}\leq \\\inf\limits_{a}
\left(a^2(\gamma_1^2+...+\gamma_{n-1}^2+1)-2a\sum_{i=1}^{n-1}\gamma_i\right)
a\biggl(\int\limits _{H_1}\left(x_n^{a\gamma-n}\right)^{\frac{r}{r-2}}~dx\biggl)^{\frac{r-2}{r}}
B^2_{r,2}(H_1),
\nonumber
\end{multline}
where $(2n)/(\gamma r)<a\leq (n-2)/(\gamma-2)$ and $B_{r,2}(H_1)$ is the best constant in the $(r,2)$-Sobolev-Poincar\'e inequality in the domain $H_1$.
}
\vskip 0.5cm

Theorem~B will be proved in Section 4.

Note, that space quasiconformal mappings with the additional assumption of local Lipschitz condition on the inverse mapping:
$$
\lim\sup\limits_{y\to x}\frac{\left|\varphi^{-1}(x)-\varphi^{-1}(y)\right|}{|x-y|}\leq K<\infty
$$ 
are weak $p$-quasiconformal mappings for all $1\leq p\leq n$. It follows from simple calculations: 
$$
\left(\frac{|D\varphi(x)|^p}{|J(x,\varphi)|}\right)^{\frac{1}{p}}=
\left(\frac{|D\varphi(x)|^n}{|J(x,\varphi)|}|D\varphi(x)|^{p-n}\right)^{\frac{1}{p}}\leq K^{\frac{n}{p}}_n(\Omega)K^{\frac{n-p}{p}}<\infty.
$$

Another example of weak $p$-quasiconformal mappings are bi-Lipschitz homeomorphisms. Recall that a homeomorphism 
$\varphi:\Omega\to\Omega'$ is called a bi-Lipschitz homeomorphism if there exists a constant $0<K<\infty$ such that
$$
\frac{1}{K}\leq \lim\sup\limits_{y\to x}\frac{\left|\varphi(x)-\varphi(y)\right|}{|x-y|}\leq K.
$$ 
Then
$$
\left(\frac{|D\varphi(x)|^p}{|J(x,\varphi)|}\right)^{\frac{1}{p}}\leq
\left({K^{p+n}}\right)^{\frac{1}{p}}<\infty.
$$

Remark that a homeomorphism $\varphi:\Omega\to\Omega'$ generates an isomorphism of Sobolev spaces $W_{p}^{1}(\Omega)$ and $W_{p}^{1}(\Omega')$, $1\leq p<n$, if and only if $\varphi$ is a bi-Lipschitz homeomorphism \cite{Mar}. 

Let us give a simple illustration of Theorem~A.  Consider the ellipse $E\subset\mathbb R^2$:
$$
E=\left\{(x,y)\in R^2 : \frac{x^2}{a^2}+\frac{y^2}{b^2}\leq 1,\,\, a\geq b\right\}.
$$
The linear mapping 
$$
\varphi_l(x,y)=\left(\begin{array}{cc}
a & 0  \\
0 & b 
\end{array} \right)
\left(\begin{array}{cc}
a \\
b 
\end{array} \right)
$$
maps the unit disc $\mathbb D$ onto $E$. By definitions
$$
K^2_2(\mathbb D)=\frac{a^2}{ab},\,\,\text{and}\,\,M^2_2(\mathbb D)=ab.
$$
Hence by the Theorem~A
$$
\mu_1(E)\geq \frac{1}{K^2_{2}(\mathbb D) M^2_2(\mathbb D)}\mu_1(\mathbb D)=\frac{(j'_{1,1})^2}{a^2},
$$
where $j'_{1,1}$ is the first positive zero of the derivative of the Bessel function $J_1$.

This estimate is better then the classical estimate for convex domains \cite{PW}
$$
\mu_1(E)\geq \frac{\pi^2}{d(E)^2},
$$
because $d(E)=2a$ and $2j'_{1,1}>\pi$.

\begin{rem} In our recent works we studied
composition operators on Sobolev spaces defined in planar domains
and induced by conformal mappings \cite{GU1}, i.e. conformal composition operators on Sobolev spaces. It permitted us to prove existence of weighted Sobolev embeddings \cite{GU2,GU3} with
universal conformal weights for simply connected planar domains. Another application of the conformal composition operators to spectral
stability problems for so-called conformal regular domains can be found in \cite{BGU1}.
\end{rem}

\section{Sobolev-Poincar\'e inequalities for functions of $L^{1}_{p}(\Omega)$}

\subsection{Composition Operators on Lebesgue Spaces}

A mapping $\varphi:\Omega\to\mathbb{R}^{n}$ is weakly differentiable on $\Omega$, if its coordinate functions have weak derivatives on $\Omega$. Hence its formal Jacobi
matrix $D\varphi(x)$ and its determinant (Jacobian) $J(x,\varphi)$
are well defined at almost all points $x\in\Omega$. The norm $|D\varphi(x)|$
of the matrix $D\varphi(x)$ is the norm of the corresponding linear
operator. We will use the same notation for this matrix and the corresponding
linear operator.

A mapping $\varphi:\Omega\to\mathbb{R}^{n}$ possesses the Luzin $N$-property if an image of any set of measure zero has measure zero. Any Lipschitz mapping possesses the
Luzin $N$-property.

The following theorem  about composition operator on Lebesgue spaces is well known (see, for example \cite{VU1}): 

\begin{thm}
\label{thm:LpLq} Let a
homeomorphism $\varphi:\Omega\to\widetilde{\Omega}$ between two domains $\Omega$ and $\widetilde{\Omega}$
be weakly differentiable. Then the
composition operator 
\[
(\varphi^{-1})^{\ast}:L_{r}(\Omega)\to L_{s}(\widetilde{\Omega}),\,\,\,1\leq s\leq r<\infty,
\]
defined by the composition rule $(\varphi^{-1})^{\ast}(g)=g\circ\varphi^{-1}$, is bounded, if and only if $\varphi$ possesses the Luzin $N$-property
and 
\begin{gather}
M_{r,s}(\Omega)=\biggl(\int\limits _{\Omega}\left|J(x,\varphi)\right|^{\frac{r}{r-s}}~dx\biggl)^{\frac{r-s}{rs}}<\infty, \,\,\,1\leq s<r<\infty,\nonumber\\
M_{s,s}(\Omega):=M_s(\Omega)=\operatorname{ess}\sup\limits_{x\in\Omega}\left|J(x,\varphi)\right|^{\frac{1}{s}}<\infty, \,\,\,1\leq s=r<\infty.
\nonumber
\end{gather}
 The norm of the composition operator $\|(\varphi^{-1})^{\ast}\|$ is equal to $M_{r,s}(\Omega)$.
\end{thm}

\subsection{Composition Operators on Sobolev Spaces}

Let $\Omega$ and $\widetilde{\Omega}$ be domains in $\mathbb R^n$. We say that
a homeomorphism $\varphi:\Omega\to\widetilde{\Omega}$ induces a bounded composition
operator 
\[
\varphi^{\ast}:L^1_p(\widetilde{\Omega})\to L^1_p(\Omega),\,\,\,1\leq p\leq\infty,
\]
by the composition rule $\varphi^{\ast}(f)=f\circ\varphi$, if for
any function $f\in L^1_p(\widetilde{\Omega})$, the composition $\varphi^{\ast}(f)\in L^1_p(\Omega)$
is defined quasi-everywhere in $\Omega$ and there exists a constant $K_{p}(\varphi;\Omega)<\infty$ such that 
\[
\|\varphi^{\ast}(f)\mid L^1_p(\Omega)\|\leq K_{p}(\varphi;\Omega)\|f\mid L^1_p(\widetilde{\Omega})\|.
\]

Let $\varphi:\Omega\to\widetilde{\Omega}$ be weakly differentiable in $\Omega$. The mapping $\varphi$ is the mapping of finite distortion if $|D\varphi(x)|=0$ for almost all $x\in Z=\{x\in\Omega : J(x,\varphi)=0\}$.

\begin{thm}
\label{CompTh} \cite{GGR} A homeomorphism $\varphi:\Omega\to\widetilde{\Omega}$
between two domains $\Omega$ and $\widetilde{\Omega}$ induces a bounded composition
operator 
\[
\varphi^{\ast}:L^1_p(\widetilde{\Omega})\to L^1_p(\Omega),\,\,\,1\leq p<\infty,
\]
 if and only if $\varphi\in W_{1,\loc}^{1}(\Omega)$, has finite distortion
and 
\[
K_{p}(\varphi;\Omega)=\left(\ess\sup\limits_{x\in\Omega}\frac{|D\varphi(x)|^p}{|J(x,\varphi)|}\right)^{\frac{1}{p}}<\infty.
\]
\end{thm}

\subsection{Sobolev-Poincar\'e inequalities}
\begin{defn} Let $1\leq r,p\leq \infty$. A bounded domain $\Omega\subset\mathbb R^n$ is called a $(r,p)$-Sobolev-Poincar\'e domain, if for any function $f\in L^1_p(\Omega)$, the $(r,p)$-Sobolev-Poincar\'e inequality
$$
\inf\limits_{c\in\mathbb R}\|f-c\mid L_r(\Omega)\|\leq B_{r,p}(D)\|\nabla f\mid L_p(\Omega)\|
$$
holds.
\end{defn}

We start from the case when weak $p$-quasiconformal mappings have a bounded Jacobian ($|J(x,\varphi)|\leq c<\infty$ for almost all $x$).  Examples of such homeomorphisms are bi-Lipschitz homeomorphisms and Lipschitz weak $p$-quasiconformal mappings.

\begin{thm}
\label{thm:PoincareEnCompPP} 
Let a bounded domain $\Omega\subset\mathbb R^n$ be a $(r,p)$-Sobolev-Poncar\'e domain, $1<p\leq r<\infty$, and there exists a weak $p$-quasiconformal mapping $\varphi: \Omega\to\widetilde{\Omega}$ of a domain $\Omega$ onto a bounded domain $\widetilde{\Omega}$ such that 
$$
M_r(\Omega)=\operatorname{ess}\sup\limits_{x\in \Omega}\left|J(x,\varphi)\right|^{\frac{1}{r}}<\infty.
$$
Then in the domain $\widetilde{\Omega}$ the $(r,p)$-Sobolev-Poincar\'e inequality 
\begin{equation}
\inf\limits_{c\in\mathbb R}\biggl(\int\limits _{\widetilde{\Omega}}|f(x)-c|^{r}~dx \biggr)^{\frac{1}{r}}\leq B_{r,p}(\widetilde{\Omega})\biggl(\int\limits _{\widetilde{\Omega}}
|\nabla f(x)|^{p}~dx\biggr)^{\frac{1}{p}},\,\, f\in W^1_p(\widetilde{\Omega}),
\label{eq:WPPI}
\end{equation} 
holds and 
$$
B_{r,p}(\widetilde{\Omega})\leq K_{p}(\Omega)M_r(\Omega) B_{r,p}({\Omega}).
$$
Here $B_{r,p}({\Omega})$ is the best constant in the $(r,p)$-Sobolev-Poincar\'e inequality in the domain $\Omega$.
\end{thm}

\begin{proof}
Let $f\in L^1_p(\widetilde{\Omega})$. By the conditions of the theorem there exists a weak $p$-quasiconformal homeomorphism 
$\varphi: \Omega\to \widetilde{\Omega}$. Hence, the composition operator 
$$
\varphi^{\ast}: L^1_p(\widetilde{\Omega})\to L^1_p(\Omega)
$$
is bounded. Because $\Omega$ is a bounded $(r,p)$-Sobolev-Poincar\'e domain  $g=\varphi^{\ast}(f)\in W^1_p(\Omega)$.

Using the change of variable formula we obtain: 

\begin{multline}
\inf\limits_{c\in \mathbb R}\biggl(\int\limits_{\widetilde{\Omega}}|f(y)-c|^r~dy\biggr)^{\frac{1}{r}}=\inf\limits_{c\in \mathbb R}
\biggl(\int\limits_{\Omega}|f(\varphi(x))-c|^r|J(x,\varphi)|~dx\biggr)^{\frac{1}{r}}\\
\leq \ess\sup\limits_{x\in \Omega}|J(x,\varphi)|^{\frac{1}{r}}
\inf\limits_{c\in \mathbb R}\biggl(\int\limits_{\Omega}|f(\varphi(x))-c|^r~dx\biggr)^{\frac{1}{r}}
=M_r(\Omega)\inf\limits_{c\in \mathbb R}\biggl(\int\limits_{\Omega}|g(x)-c|^r~dx\biggr)^{\frac{1}{r}}.
\nonumber
\end{multline}

Because the domain $\Omega$ is a $(r,p)$-Sobolev-Poincar\'e domain we have
$$
\inf\limits_{c\in \mathbb R}\biggl(\int\limits_{\Omega}|g(x)-c|^r~dx\biggr)^{\frac{1}{r}}\leq B_{r,p}(\Omega)\biggl(\int\limits_{\Omega}|\nabla g(x)|^p~dx\biggr)^{\frac{1}{p}}.
$$
Hence
$$
\inf\limits_{c\in \mathbb R}\biggl(\int\limits_{\widetilde{\Omega}}|f(y)-c|^r~dy\biggr)^{\frac{1}{r}}\leq
M_r(\Omega)B_{r,p}(\Omega) \|g\mid L^1_p(\Omega)\|.
$$
By Theorem \ref{CompTh} 
$$
\|g\mid L^1_p(\Omega)\| \leq K_{p}^{\frac{1}{p}}(\Omega)\|f\mid L^1_p(\widetilde{\Omega})\|.
$$

Therefore
$$
\inf\limits_{c\in \mathbb R}\biggl(\int\limits_{\widetilde{\Omega}}|f(y)-c|^r~dy\biggr)^{\frac{1}{r}}\leq
K_p(\Omega)M_r(\Omega)B_{r,p}(\Omega)\biggl(\int\limits _{\widetilde{\Omega}}
|\nabla f(x)|^{p}~dx\biggr)^{\frac{1}{p}}.
$$
\end{proof}

\begin{cor} \label{Cor1} Under conditions of Theorem~\ref{thm:PoincareEnCompPP} the embedding operator
$$
i: W^1_p(\widetilde{\Omega}) \hookrightarrow L_p(\widetilde{\Omega})
$$
is compact.
\end{cor}
It follows immediately by the generalized version of Rellich-Kondrachov compactness theorem (see, for example, \cite{M} or \cite{HK}) and the $(r,p)$--Sobolev-Poincar\'e inequality for $r>p$.

Hence, the standard corollary of Theorem~\ref{thm:PoincareEnCompPP} is a conclusion about a discrete spectral structure and a lower estimate of the first non-trivial eigenvalue $\mu_1(\widetilde{\Omega})$ of the spectral Neumann problem for the Laplace operator in $\widetilde{\Omega}$ (via the first non-trivial eigenvalue $\mu_1({\Omega})$):

\vskip 0.5cm

{\bf Theorem A.} {\it  
Suppose that there exists a weak $2$-quasiconformal mapping $\varphi: \Omega\to\widetilde{\Omega}$, of a bounded Lipschitz domain 
$\Omega\subset\mathbb R^n$ onto $\widetilde{\Omega}$, such that 
$$
M_2(\Omega)=\operatorname{ess}\sup\limits_{x\in \Omega}\left|J(x,\varphi)\right|^{\frac{1}{2}}<\infty.
$$
Then the spectrum of Neumann-Laplace operator in $\widetilde{\Omega}$ is discrete, can be written in the form
of a non-decreasing sequence
\[
0=\mu_0(\widetilde{\Omega})<\mu_{1}(\widetilde{\Omega})\leq\mu_{2}(\widetilde{\Omega})\leq...\leq\mu_{n}(\widetilde{\Omega})\leq...\,,
\]
and
\begin{equation}
\label{thm:estpp}
\frac{1}{\mu_1(\widetilde{\Omega})}\leq K^2_{2}(\Omega) M^2_2(\Omega)\frac{1}{\mu_1(\Omega)}.
\end{equation}
}
\begin{proof}
Because  $\Omega$ is a bounded Lipschitz domain the embedding operator
$$
i: W^1_2(\widetilde{\Omega}) \hookrightarrow L_2(\widetilde{\Omega})
$$
is compact, $\Omega$ is a $(r,2)$-Sobolev-Poincar\'e domain for $2 \leq r < \frac{2n}{n-2}$ and  the spectrum of the Neumann-Laplace operator is discrete. The condition of Theorem \ref{thm:PoincareEnCompPP} are satisfied for domains $\Omega$ and $\widetilde \Omega$. Therefore by Corollary \ref{Cor1} the embedding operator 
$$
i: W^1_2(\widetilde{\Omega}) \hookrightarrow L_2(\widetilde{\Omega})
$$
is compact. Hence the spectrum of Neumann-Laplace operator in $\widetilde{\Omega}$ is discrete and can be written in the form
of a non-decreasing sequence
\[
0=\mu_0(\widetilde{\Omega})<\mu_{1}(\widetilde{\Omega})\leq\mu_{2}(\widetilde{\Omega})\leq...\leq\mu_{n}(\widetilde{\Omega})\leq...\,.
\]
Because $\mu_1(\widetilde\Omega)^{-\frac{1}{2}}$ is the best constant $B_{2,2}(\widetilde \Omega)$ we have finally by Theorem~\ref{thm:PoincareEnCompPP}
\begin{equation}
\frac{1}{\mu_1(\widetilde{\Omega})}\leq K^2_{2}(\Omega) M^2_2(\Omega)\frac{1}{\mu_1(\Omega)}.
\end{equation} 

\end{proof}

\vskip 0.5cm

Boundedness of a Jacobian of a weak $p$-quasiconformal mapping is a sufficient but restrictive assumption. In this case $\varphi$ is a Lipschitz mapping and as result  an image of a Lipschitz domain can not be a domain with external singularities. We shall use weak $p$-quasiconformal mappings with an integrable Jacobian that represent a more flexible class of mappings, which allows us map Lipschitz domains onto cusp domains. 

 In the study of spectral stability of Dirichlet-Laplacian in planar simply connected domains we introduced a notion of conformal regular domains \cite{BGU1} for which Jacobian $J(x,\varphi)$ of a conformal mapping $\varphi:\mathbb D\to\widetilde{\Omega}$ of the unit disc $\mathbb D\subset\mathbb R^n$ onto $\Omega$ is integrable in  some degree $\alpha>1$. We used this class of domains for the spectral estimates of the first nontrivial Neumann eigenvalues for Laplacian \cite{GU15}.

In the space case we suppose (by an analogy with the plane case) that Jacobians $J(x,\varphi)$ of weak $p$-quasiconformal mapping $\varphi:\Omega\to\widetilde{\Omega}$ of a $(r,p)$-Sobolev-Poncar\'e domain $\Omega$ onto $\widetilde{\Omega}$  are integrable in some degree $\alpha>1$. Under this regularity condition on $\widetilde{\Omega}$ the following statement is correct:

\begin{thm}
\label{thm:PoincareEnCompPQ} 
Let a bounded domain $\Omega\subset\mathbb R^n$ be a $(r,p)$-Sobolev-Poncar\'e domain, $1<p\leq r<\infty$, and there exists a weak $p$-quasiconformal homeomorphism $\varphi: \Omega\to\widetilde{\Omega}$ of a domain $\Omega$ onto a bounded domain $\widetilde{\Omega}$ such that
$$ 
M_{r,s}(\Omega)=\biggl(\int\limits _{\Omega}\left|J(x,\varphi)\right|^{\frac{r}{r-s}}~dx\biggl)^{\frac{r-s}{rs}}<\infty
$$
for some $s<r$.
Then in the domain $\widetilde{\Omega}$ the $(s,p)$-Sobolev-Poincar\'e inequality
\begin{equation}
\biggl(\int\limits _{\widetilde{\Omega}}|f(x)-f_{\widetilde{\Omega}}|^{s}~dx \biggr)^{\frac{1}{s}}\leq B_{s,p}(\widetilde{\Omega})\biggl(\int\limits _{\widetilde{\Omega}}
|\nabla f(x)|^{p}~dx\biggr)^{\frac{1}{p}},\,\, f\in W^1_p(\widetilde{\Omega}),
\label{eq:WPI}
\end{equation} 
holds and 
$$
B_{s,p}(\widetilde{\Omega})\leq K_{p}(\Omega)M_{r,s}(\Omega)B_{r,p}(\Omega).
$$
Here $B_{r,p}(\Omega)$ is the best constant in the $(r,p)$-Sobolev-Poincar\'e inequality in the domain $\Omega$.

\end{thm}

\begin{proof}
Let $f\in L^1_p(\widetilde{\Omega})$. By the conditions of the theorem there exists a $p$-quasiconformal homeomorphism $\varphi: \Omega\to \widetilde{\Omega}$. By Theorem \ref{CompTh} the composition operator 
$$
\varphi^{\ast}: L^1_p(\widetilde{\Omega})\to L^1_p(\Omega)
$$
is bounded. Because the bounded domain $\Omega$ is a $(r,p)$-Sobolev-Poncar\'e domain  $g=\varphi^{\ast}(f)\in W^1_p(\Omega)$.

Let $s\geq 1$. Using the change of variable formula and the H\"older inequality we obtain: 

\begin{multline}
\inf\limits_{c\in \mathbb R}\biggl(\int\limits_{\widetilde{\Omega}}|f(y)-c|^s~dy\biggr)^{\frac{1}{s}}=\inf\limits_{c\in \mathbb R}
\biggl(\int\limits_{\Omega}|f(\varphi(x))-c|^s|J(x,\varphi)|~dx\biggr)^{\frac{1}{s}}\\
\leq \biggl(\int\limits_{\Omega}|J(x,\varphi)|^{\frac{r}{r-s}}~dx\biggr)^{\frac{r-s}{rs}}
\inf\limits_{c\in \mathbb R}\biggl(\int\limits_{\Omega}|f(\varphi(x))-c|^r~dx\biggr)^{\frac{1}{r}}\\
=M_{r,s}(\Omega)\inf\limits_{c\in \mathbb R}\biggl(\int\limits_{\Omega}|g(x)-c|^r~dx\biggr)^{\frac{1}{r}}.
\nonumber
\end{multline}

 Because the domain $\Omega$ is a $(r,p)$-Sobolev-Poincar\'e domain the following inequality holds:
$$
\inf\limits_{c\in \mathbb R}\biggl(\int\limits_{\Omega}|g(x)-c|^r~dx\biggr)^{\frac{1}{r}}\leq B_{r,p}(\Omega)\biggl(\int\limits_{\Omega}|\nabla g(x)|^p~dx\biggr)^{\frac{1}{p}}.
$$
Combining two previous inequalities we have
$$
\inf\limits_{c\in \mathbb R}\biggl(\int\limits_{\widetilde{\Omega}}|f(y)-c|^s~dy\biggr)^{\frac{1}{s}}\leq
M_{r,s}(\Omega)B_{r,p}(\Omega) \|g\mid L^1_p(\Omega)\|.
$$
By Theorem \ref{CompTh} 
$$
\|g\mid L^1_p(\Omega)\| \leq K_{p}(D)\|f\mid L^1_p(\widetilde{\Omega})\|.
$$

Finally we obtain 
$$
\inf\limits_{c\in \mathbb R}\biggl(\int\limits_{\widetilde{\Omega}}|f(y)-c|^s~dy\biggr)^{\frac{1}{s}}\leq 
K_{p}(\Omega)M_{r,s}(\Omega) B_{r,p}(\Omega)\biggl(\int\limits_{\widetilde{\Omega}}|\nabla f|^p~dy\biggr)^{\frac{1}{p}}.
$$

It means that 
$$
B_{s,p}(\widetilde{\Omega})\leq K_{p}(\Omega)M_{r,s}(\Omega)B_{r,p}(\Omega).
$$

\end{proof}

\vskip 0.5cm

\begin{cor} \label{Cor2} Under conditions of Theorem~\ref{thm:PoincareEnCompPQ} the embedding operator
$$
i: W^1_p(\widetilde{\Omega}) \hookrightarrow L_p(\widetilde{\Omega})
$$
is compact.
\end{cor}
It follows immediately by the generalized version of Rellich-Kondrachov compactness theorem (see, for example, \cite{M} or \cite{HK}) and the $(r,p)$--Sobolev-Poincar\'e inequality for $r>p$. 

\vskip 0.5cm

We are ready to establish the main lower estimate:

{\bf Theorem C.}
{\it 
Let a domain $\Omega\subset\mathbb R^n$ be a $(r,2)$-Sobolev-Poncar\'e domain, $r>2$, and there exists a weak $2$-quasiconformal homeomorphism $\varphi: \Omega\to\widetilde{\Omega}$ of a domain $\Omega$ onto a bounded domain $\widetilde{\Omega}$ such that
$$ 
M_{r,2}(\Omega)=\biggl(\int\limits _{\Omega}\left|J(x,\varphi)\right|^{\frac{r}{r-2}}~dx\biggl)^{\frac{r-2}{2r}}<\infty
$$
for some $r>2$.
Then the spectrum of Neumann-Laplace operator in $\widetilde{\Omega}$ is discrete, can be written in the form
of a non-decreasing sequence
\[
0=\mu_0(\widetilde{\Omega})<\mu_{1}(\widetilde{\Omega})\leq\mu_{2}(\widetilde{\Omega})\leq...\leq\mu_{n}(\widetilde{\Omega})\leq...\,,
\]
and
$$
\frac{1}{\sqrt{\mu_1(\widetilde{\Omega})}}\leq K_{2}(\Omega)M_{r,2}(\Omega)B_{r,2}(\Omega),
$$
where $B_{r,2}(\Omega)$ is the best constant in the $(r,p)$-Sobolev-Poincar\'e inequality for the domain $\Omega$.
}
\vskip 0.5cm

The proof is the same as for Theorem A. We only need to refer Corollary \ref{Cor2} instead of Corollary
\ref{Cor1} and Theorem \ref{thm:PoincareEnCompPQ} instead of Theorem \ref{thm:PoincareEnCompPP}.

\subsection{Simple Examples}

We give two simple examples of weak $2$-quasiconformal mappings and its applications to lower estimates of the first non-trivial Neumann eigenvalue of the Laplace operator based on Theorem A.

Consider the linear invertible map  $\varphi_l:\mathbb{R}^2\to \mathbb{R}^2$ 
$$
\varphi_l(x,y)=\left(\begin{array}{cc}
a & 0  \\
0 & b 
\end{array} \right)
\left(\begin{array}{cc}
x \\
y 
\end{array} \right)
$$
Then
$$
K^2_2(\varphi_l)=\frac{\max\{a^2,b^2\}}{ab}\,\,\text{and}\,\, J(x,\varphi_l)=ab.
$$

The first example demonstrate that Theorem A is exact for rectangles. 
\begin{exa}
Let $\Omega=\mathbb Q=(0,1)\times(0,1)$ be the unit square, then $\varphi_l(\mathbb Q)=\mathbb P=(0,a)\times(0,b)$.
Then by Theorem~A
$$
\mu_1(\mathbb P)\geq \frac{1}{\frac{\max\{a^2,b^2\}}{ab}\cdot ab}\mu_1(\mathbb Q)=\frac{\pi^2}{\max\{a^2,b^2\}}.
$$
\end{exa}

This estimate coincides with the known exact value of $\mu_1$.

The second example is an ellipse. Our estimate is better than classical one as we explained in Introduction.

\begin{exa}
Let $\Omega=\mathbb D$ be the unit disc, then 
$$
\varphi_l(\mathbb D)=E=\left\{(x,y)\in R^2 : \frac{x^2}{a^2}+\frac{y^2}{b^2}\leq 1\right\}.
$$
Then by Theorem~A
$$
\mu_1(E)\geq \frac{1}{\frac{\max\{a^2,b^2\}}{ab}\cdot ab}\mu_1(\mathbb D)=\frac{(j'_{1,1})^2}{\max\{a^2,b^2\}}.
$$
\end{exa}

IN next section Theorem C will be applied to domains with anisotropic H\"older singularities.

\section{Spectral estimates in domains with H\"older singularities}

Using the classical technique \cite{GT}  we  obtained in \cite{GU15} the following estimate of the Poincar\'e constants for bounded convex domains :

\begin{lem}(\cite{GU15}, Proposition 4.7)
\label{EstCon}
Let $\Omega\subset\mathbb R^n$ be a bounded convex domain. Then
\begin{equation}
B_{q,p}(\Omega)\leq
\frac{diam(\Omega)^n}{n|\Omega|}\left(\frac{1-\delta}{{1}/{n}-\delta}\right)^{1-\delta}\omega_n^{1-\frac{1}{n}}|\Omega|^{\frac{1}{n}-\delta},\,\,\delta=\frac{1}{p}+\frac{1}{q}\geq 0.
\nonumber
\end{equation}
\end{lem}

Let us remarks that images of  convex Lipschitz domains under weak $p$-quasiconformal mappings are not necessary convex and  Lipschitz. 

One of possible examples is a class of domains $H_g$ with anisotropic H\"older singularities that will be defined below. 

As a basic convex Lipschitz domain we choose the standard $n$-dimensional simplex $H_1$. By elementary calculations $\diam H_1=1$ and 
$|H_1|={1}/{n}$. Lemma \ref{EstCon} leads to the following estimate of the Sobolev-Poincar\'e constant for $H_1$:

\begin{equation}
\label{H1}
B_{q,p}(H_1)\leq
n\left(\frac{1-\delta}{{1}/{n}-\delta}\right)^{1-\delta}\omega_n^{1-\frac{1}{n}}\left(\frac{1}{(n+1)!}\right)^{\frac{1}{n}-\delta},\,\,\delta=\frac{1}{p}+\frac{1}{q}\geq 0
\end{equation}

 Define domains $H_g$  with anisotropic H\"older singularities (introduced in \cite{GGu}):
$$
H_g=\{ x\in\mathbb R^n : 0<x_n<1, 0<x_i<g_i(x_n),
\,i=1,2,\dots,n-1\}.
$$
Here $g_i(\tau)=\tau^{\gamma_i}$, $\gamma_i\geq 1$, $0\leq\tau\leq 1$ are H\"older functions and for the function $G=\prod_{i=1}^{n-1}g_i$ denote by
$$
\gamma=\frac{\log G(\tau)}{\log \tau}+1.
$$
It is evident that $\gamma\geq n$. In the case $g_1=g_2=\dots=g_{n-1}$ we will say that domain $H_g$ is a domain with $\sigma$-H\"older singularity, $\sigma=(\gamma-1)/(n-1)$.
For $g_1(\tau)=g_2(\tau)=\dots=g_{n-1}(\tau)=\tau$ we will use notation $H_1$ instead
of $H_g$.

The mapping $\varphi_a: H_1\to H_g$, $a>0$,
$$
\varphi_a(x)=\left(\frac{x_1}{x_n}g^a_1(x_n),\dots,\frac{x_{n-1}}{x_n}g^a_{n-1}(x_n),x_n^a\right).
$$
is a map from the Lipschitz convex domain $H_1$ onto the "cusp"' domain $H_g$.
 
By simple calculations 
$$
\frac{\partial(\varphi_a)_i}{\partial
x_i}=\frac{g^a_i(x_n)}{x_n},\quad
\frac{\partial(\varphi_a)_i}{\partial
x_n}=\frac{-x_ig^a_i(x_n)}{x_n^{2}}+\frac{ax_ig^{a-1}_i(x_n)}{x_n}g'_i(x_n)
\quad\text{and}\quad\frac{\partial(\varphi_a)_n}{\partial
x_n}=ax_n^{a-1}
$$
for any $i=1,...,n-1$. Hence $J(x,\varphi_a)=ax_n^{a-n}G^a(x_n)=ax_n^{a\gamma-n}$, $J(x,\varphi_a)\leq a$ for $a>1$ and
\begin{multline}\label{maps}
D\varphi_a (x)=
\left(\begin{array}{cccc}
x_n^{a\gamma_1-1} & 0 & ... & (a\gamma_1-1)x_1x_n^{a\gamma_1-2}\\
0 & x_n^{a\gamma_2-1} & ... & (a\gamma_2-1)x_2x_n^{a\gamma_2-2}\\
... & ... & ... & ...\\
0 & 0 & ... & ax_n^{a-1}
\end{array} \right)\\
=
x_n^{a-1}\left(\begin{array}{cccc}
x_n^{a\gamma_1-a} & 0 & ... & (a\gamma_1-1)\frac{x_1}{x_n}x_n^{a(\gamma_1-1)}\\
0 & x_n^{a\gamma_2-a} & ... & (a\gamma_2-1)\frac{x_2}{x_n}x_n^{a(\gamma_2-1)}\\
... & ... & ... & ...\\
0 & 0 & ... & a
\end{array} \right).
\end{multline}

Because $0<x_n<1$ and $x_1/x_n<1$ we have the following estimate
\begin{multline}
|D\varphi_a(x)|\leq x_n^{a-1}\sqrt{\sum_{i=1}^{n-1}(a\gamma_i-1)^2+n-1+a^2}\\
= x_n^{a-1}\sqrt{a^2(\gamma_1^2+...+\gamma_{n-1}^2+1)-2a\sum_{i=1}^{n-1}\gamma_i}:=A_a(\gamma) x_n^{a-1}.
\nonumber
\end{multline}

We used a short notation $A_a(\gamma)$ for the square root in the right hand side.
Then 
$$
\frac{|D\varphi_a(x)|^p}{J(x,\varphi_a)}\leq A^p_a(\gamma) x_n^{p(a-1)-(a\gamma-n)}\leq K_p^{p}<\infty
$$
if $p(a-1)-(a\gamma-n)\geq 0$. 

Therefore, in the case $a>1$ the mapping $\varphi_a: H_1\to H_g$ is a weak $p$-quasiconformal mapping for any $p\geq (a\gamma -n)/(a-1)$ and 
$$
K_p(H_1)\leq A_p(\gamma)=\sqrt{a^2(\gamma_1^2+...+\gamma_{n-1}^2+1)-2a\sum_{i=1}^{n-1}\gamma_i}.
$$

We proved the following

\begin{lem}
\label{lemhol}
Let $1<p<n$. The homeomorphism $\varphi_a: H_1\to H_g$ is a weak $p$-quasiconformal homeomorphism if $0<a\leq (n-p)/(\gamma-p)$. 

\end{lem}

We are ready to prove as a consequence of Theorem C:

\vskip 0.3cm

{\bf Theorem B.} 
\label{thm:Holder}
{\it  
Let 
$$
H_g:=\{ x\in\mathbb R^n : n \geq 3, 0<x_n<1, 0<x_i<x_n^{\gamma_i},
\,i=1,2,\dots,n-1\}
$$  
$\gamma_i \geq 1$, $\gamma:=1+\sum_{i=1}^{n-1}\gamma_i$, $g:=(\gamma_1,...,\gamma_{n-1})$ .

Then the spectrum of the Neumann-Laplace operator in the domain $H_g$ is discrete, can be written in the form
of a non-decreasing sequence
\[
0=\mu_0(H_g)<\mu_{1}(H_g)\leq\mu_{2}(H_g)\leq...\leq\mu_{n}(H_g)\leq...\,,
\]
and for any $r>2$ the following inequality holds:
\begin{multline}
\frac{1}{\mu_1(H_g)} \leq K^2_{2}(H_1)M^2_{r,2}(H_1)B^2_{r,2}(H_1)
\\ \leq \inf\limits_{a}
\left(a^2(\gamma_1^2+...+\gamma_{n-1}^2+1)-2a\sum_{i=1}^{n-1}\gamma_i\right)
a\biggl(\int\limits _{H_1}\left(x_n^{a\gamma-n}\right)^{\frac{r}{r-2}}~dx\biggl)^{\frac{r-2}{r}}
B^2_{r,2}(H_1),
\nonumber
\end{multline}
where $(2n)/(\gamma r)<a\leq (n-2)/(\gamma-2)$ and $B_{r,2}(H_1)$ is the best constant in the $(r,2)$-Sobolev-Poincar\'e inequality in the domain $H_1$.
}
\vskip 0.5cm

\begin{rem} By the estimate \ref{H1} 
$$
B_{r,2}(H_1)\leq
n\left(\frac{1-\delta}{{1}/{n}-\delta}\right)^{1-\delta}\omega_n^{1-\frac{1}{n}}\left(\frac{1}{(n+1)!}\right)^{\frac{1}{n}-\delta},\,\,\delta=\frac{1}{r}+\frac{1}{2}\geq 0.
$$
\nonumber
\end{rem}
\begin{proof}
The homeomorphism $\varphi_a: H_1\to H_g$, $0<a\leq (n-2)/(\gamma-2)$,
$$
\varphi_a(x)=\left(\frac{x_1}{x_n}g^a_1(x_n),\dots,\frac{x_{n-1}}{x_n}g^a_{n-1}(x_n),x_n^a\right).
$$
maps the convex Lipschitz domain $H_1$ onto the cusp domain $H_g$ and by Lemma~\ref{lemhol} it is a weak $2$-quasiconformal homeomorphism.

Let us check conditions of Theorem C. Because $\varphi$ is a weak $2$-quasiconformal mapping its $p$-dilatation $K_2$ is bounded. The basic domain $H_1$  is an $(r,2)$-Sobolev-Poincar\'e domain, i.e. $B_{r,2}(H_1)<\infty$. We only need to estimate the constant $M_{r,2}(H_1)$. 

Here a corresponding calculations:

\begin{multline*} 
M_{r,2}(H_1)=\left(\int\limits _{H_1}\left|J(x,\varphi_a)\right|^{\frac{r}{r-2}}~dx\right)^{\frac{r-2}{2r}}\leq 
a^{\frac{1}{2}}\left(\int\limits _{H_1} \left(x_n^{a\gamma-n}\right)^{\frac{r}{r-2}}~dx\right)^{\frac{r-2}{2r}}\\
=
a^{\frac{1}{2}}\left(\int\limits _{0}^1 \left(x_n^{a\gamma-n}\right)^{\frac{r}{r-2}}\left(\int\limits_0^{x_n}~dx_1~\dots ~ 
\int\limits_0^{x_n}~dx_{n-1}\right)~dx_n\right)^{\frac{r-2}{2r}}\\
=
a^{\frac{1}{2}}\left(\int\limits _{0}^1 \left(x_n^{a\gamma-n}\right)^{\frac{r}{r-2}}\cdot x_n^{n-1}~dx_n\right)^{\frac{r-2}{2r}}.
\end{multline*}

It means that $M_{r,2}(H_1)$ is finite if 

$$
\frac{(a\gamma-n)r}{r-2}+n-1>-1,\,\,\,\text{i.~e.}\,\,\,a>\frac{2n}{\gamma r}.
$$

The conditions of  Theorem~C is fulfilled. Therefore, the spectrum of Neumann-Laplace operator in $H_g$ is discrete, can be written in the form
of a non-decreasing sequence
\[
0=\mu_0(H_g)<\mu_{1}(H_g)\leq\mu_{2}(H_g)\leq...\leq\mu_{n}(H_g)\leq...\,,
\]
and
\begin{multline}
\frac{1}{\mu_1(H_g)}\leq K^2_{2}(H_1)M^2_{r,2}(H_1)B^2_{r,2}(H_1)
\\
\leq 
\left(a^2(\gamma_1^2+...+\gamma_{n-1}^2+1)-2a\sum_{i=1}^{n-1}\gamma_i\right)
a\biggl(\int\limits _{H_1}\left(x_n^{a\gamma-n}\right)^{\frac{r}{r-2}}~dx\biggl)^{\frac{r-2}{r}}
B^2_{r,2}(H_1),
\nonumber
\end{multline}
where $(2n)/(\gamma r)<a\leq (n-2)/(\gamma-2)$.
\end{proof}

\section{Lower estimates for the first non-trivial eigenvalue of the $p$-Laplace operator}

In this section we consider the nonlinear Neumann eigenvalue problem
for the $p$-Laplace operator ($p>1$):
\begin{equation}
\label{pLaplace}
\begin{cases}
-\operatorname{div}\left(|\nabla u|^{p-2}\nabla u\right)=\mu_p|u|^{p-2}u& \text{in}\,\,\,\,\,\Omega\\
\frac{\partial u}{\partial n}=0& \text{on}\,\,\,\partial\Omega.
\end{cases}
\end{equation}

This formulation of the eigenvalue problem is correct for Lipschitz domains.

Because we are working with more general class of domains we shall use a corresponding weak formulation that coincides with the classical one for Lipschitz domains.

The weak statement of this spectral problem is as follows: a function
$u$ solves the previous problem iff $u\in W^{1,p}(\Omega)$ and 
\[
\int\limits _{\Omega} \left(|\nabla u(x)|^{p-2}\nabla u(x)\right)\cdot\nabla v(x)~dx=\mu_p \int\limits _{\Omega}|u|^{p-2}u(x)v(x)~dx
\]
for all $v\in W^{1,p}(\Omega)$. 

The first nontrivial Neumann eigenvalue $\mu_{1,p}$ can be characterized as
$$
\mu_{1,p}(\Omega)=\min\left\{\frac{\int\limits_{\Omega}|\nabla u(x)|^p~dx}{\int\limits_{\Omega}|u(x)|^p~dx}: u\in W^{1}_{p}(\Omega)\setminus\{0\},\,
\int\limits_{\Omega}|u|^{p-2}u~dx=0\right\}.
$$

Moreover, $\mu_{1,p}(\Omega)^{-\frac{1}{p}}$ is the best constant $B_{p,p}(\Omega)$ (see, for example, \cite{BCT15}) in the following Poincar\'e inequality 
$$
\inf\limits_{c\in\mathbb R}\|f-c \mid L_p(\Omega)\|\leq B_{p,p}(\Omega)\|\nabla f \mid L_p(\Omega)\|,\,\,\, f\in W^{1}_{p}(\Omega).
$$

The Theorem~\ref{thm:PoincareEnCompPP} immediately implies the following lower estimate for $\mu_{1,p}(\tilde{\Omega})$:

\begin{thm} 
\label{thm:estimate}
Suppose that there exists a $p$-quasiconformal homeomorphism  $\varphi: \Omega\to\widetilde{\Omega}$, of a bounded Lipschitz domain 
$\Omega\subset\mathbb R^n$ onto $\widetilde{\Omega}$, such that 
$$
M_p(\Omega)=\operatorname{ess}\sup\limits_{x\in \Omega}\left|J(x,\varphi)\right|^{\frac{1}{p}}<\infty.
$$
Then 
$$
\frac{1}{\mu_{1,p}(\widetilde{\Omega})}\leq K^p_{p}(\Omega) M^p_p(\Omega)\frac{1}{\mu_{1,p}(\Omega)}.
$$
\end{thm}

In the case of a corresponding integrability of Jacobian $J(x,\varphi$ in a corresponding degree we have by Theorem~\ref{thm:PoincareEnCompPQ}:

\begin{thm}
\label{thm:estrp}
Let a domain $\Omega\subset\mathbb R^n$ be a $(r,p)$-Sobolev-Poncar\'e domain, $p>1$, and there exists a weak $p$-quasiconformal homeomorphism $\varphi: \Omega\to\widetilde{\Omega}$ of a domain $\Omega$ onto a bounded domain $\widetilde{\Omega}$ such that
$$ 
M_{r,p}(\Omega)=\biggl(\int\limits _{\Omega}\left|J(x,\varphi)\right|^{\frac{r}{r-p}}~dx\biggl)^{\frac{r-p}{rp}}<\infty
$$
for some $r>p$.
Then  
$$
\frac{1}{\sqrt[p]{\mu_{1,p}(\widetilde{\Omega})}}\leq K_{p}(\Omega)M_{r,p}(\Omega)B_{r,p}(\Omega).
$$
\end{thm}

As an example, consider the linear mapping $\varphi_l:\Omega\to\widetilde{\Omega}$ of space domains $\Omega,\widetilde{\Omega}$:
$$
\varphi_l(x)=\left(\begin{array}{cccc}
a_1 & 0 & ... & 0 \\
0 & a_2 & ... & 0 \\
... & ... & ... & ... \\
0 & ... & 0 & a_n
\end{array} \right)
\left(\begin{array}{cc}
x_1 \\
x_2 \\
... \\
x_n
\end{array} \right)
$$
Then
$$
K^p_p(\Omega)=\frac{\left(\max\{a_1^2,..., a_n^2\}\right)^{\frac{p}{2}}}{a_1\cdot ... \cdot a_n}\,\,\text{and}\,\, J(x,\varphi)=a_1\cdot ... \cdot a_n.
$$

\begin{exa}
Let $\Omega=\mathbb Q^n=(0,1)\times ... \times (0,1)$ be the unit cube, then 
$\varphi_l(\mathbb Q)=\mathbb P=(0,a_1)\times ... \times (0,a_n)$ is a parallelepiped.
By Theorem~\ref{thm:estimate} the following estimate
$$
\mu_{1,p}(\mathbb P)\geq \frac{1}{\frac{\left(\max\{a_1^2,..., a_n^2\}\right)^{\frac{p}{2}}}{a_1\cdot ... \cdot a_n}\cdot a_1\cdot ... \cdot a_n}\mu_p(\mathbb Q)
=\frac{1}{\left(\max\{a_1^2,..., a_n^2\}\right)^{\frac{p}{2}}}\mu_{1,p}(\mathbb Q^n)
$$
\end{exa}
is correct.

To the best of our knowledge the exact value of $\mu_{1,p}(\mathbb Q)$ is unknown. This example gives a rate of changing of the eigenvalue. The same remark is correct for our second example also. 

\begin{exa}
Let $\Omega=\mathbb B^n$ be the unit ball, then 
$$
\varphi_l(\mathbb B)=E=\left\{x\in R^n : \frac{x_1^2}{a_1^2}+...+\frac{x_n^2}{a_n^2}\leq 1\right\}.
$$
is an ellipsoid. 

By the Theorem~\ref{thm:estimate} the following estimate 
$$
\mu_{1,p}(E)\geq \frac{1}{\frac{\left(\max\{a_1^2,..., a_n^2\}\right)^{\frac{p}{2}}}{a_1\cdot ... \cdot a_n}\cdot a_1\cdot ... \cdot a_n}\mu_p(\mathbb B)
=\frac{1}{\left(\max\{a_1^2,..., a_n^2\}\right)^{\frac{p}{2}}}\mu_{1,p}(\mathbb B).
$$
is correct.
\end{exa}

\begin{rem} In the case of the first non-trivial Neumann eigenvalue for the Laplacian we have
$$
\mu_1(E)\geq \frac{p_{n/2}}{\max\{a_1^2,..., a_n^2\}},
$$
where $p_{n/2}$ denotes the first positive zero of the function $(t^{1-n/2}J_{n/2}(t))'$.
\end{rem}

\section{Appendix: On weak $p$-quasiconformal mappings}

In the first part of this section section we collect basic facts about weak $p$-quasiconformal mappings from our previous works. This class of mappings arises as a natural generalization of quasiconformal mappings.  For $n$-dimensional domains  and  $p=n$ this class coincides with quasiconformal mappings.

In the second part is devoted on the case $p=n-1$. We are focused on this case, because in $3$-dimensional case this class was used for spectral problems of Neumann-Laplace operators. Results of this subsections are new but based on our previous works. 

Unfortunately basic properties of weak $2$-quasiconformal homeomorphisms in dimension more then $3$ are much less known and represent open problems.

\subsection{Generalized quasiconformal mappings}

Let $\Omega\subset\mathbb R^n$ be an open set and $\varphi$ is a weak $p$-quasiconformal homeomorphism. Because $\varphi \in W^1_{1,loc}$ the formal Jacobi matrix
$D\varphi(x)=\left(\frac{\partial \varphi_i}{\partial x_j}(x)\right)$, $i,j=1,\dots,n$, and its determinant (Jacobian) $J(x,\varphi)=\det D\varphi(x)$ are well defined at almost all points $x\in \Omega$. The norm $|D\varphi(x)|$ of the matrix
$D\varphi(x)$ is the norm of the corresponding linear operator $D\varphi (x):\mathbb R^n \rightarrow \mathbb R^n$ defined by the matrix $D\varphi(x)$.

For a mapping $\varphi: \Omega\to \mathbb R^n$ of the class $L^1_{p,\loc}(\Omega)$ we define the local $p$-distortion
$$
K_p(x)=\inf\{ k: |D\varphi(x)|\leq k |J(x,\varphi)|^{\frac{1}{p}}, \,\,x\in D\}.
$$
A mapping $\varphi: \Omega\to \mathbb R^n$ of the class $L^1_{p,\loc}(\Omega)$ has a finite distortion if the dilatation function $K_p$ is well defined at almost all points $x\in \Omega$, i.~e.  $D\varphi(x)=0$ for almost all points $x$
that belongs to set $Z=\{x\in \Omega:J(x,\varphi)=0\}$.

Necessity of studying of Sobolev mappings with finite distortion arises in problems of the non-linear elasticity theory \cite{B1, B2}.
In these works J.~M.~Ball introduced classes of mappings, defined on bounded domains $\Omega \in \mathbb R^n$:
$$
A^+_{p,q}(\Omega)=\{\varphi\in W^1_p(\Omega) : \adj D\varphi\in L_q(\Omega),\quad J(x,\varphi)>0 \quad\text{a.~e. in}\quad \Omega\},
$$
$p,q>n$, where $\adj D\varphi$ is the formal adjoint matrix to the Jacobi matrix $D\varphi$:
$$
\adj D\varphi(x)\cdot D\varphi(x) = \operatorname{Id} J(x,\varphi).
$$

By definition a weak $p$-quasiconformal homeomorphism  homeomorphism $\varphi:\Omega\to\widetilde{\Omega}$ has finite distortion
and the local $p$-dilatation
$$
K_p(\Omega)=\ess\sup\limits_{x\in\Omega}K_p(x):=\left(\ess\sup\limits_{x\in\Omega}\frac{|D\varphi(x)|^p}{|J(x,\varphi)|}\right)^{\frac{1}{p}}
$$
is bounded. We called the quantity $K_p(\Omega)$ the weak $p$-quasiconformalty coefficient (dilatation). 

We need the notion of the $p$-capacity for a further description of the weak $p$-quasiconformal mappings. Let a domain $\Omega\subset\mathbb R^n$ and $F_{0}$, $F_{1}$ be two disjoint compact subset of $\Omega$. We call the triple $E=(F_{0},F_{1};\Omega)$ a condenser.

The value \[
\cp_p(E)=\cp_p(F_{0},F_{1};\Omega)=\inf\int\limits _{\Omega}|\nabla v|^{p}~dx,
\]
where the infimum is taken over all non negative functions $v\in C(\Omega)\cap L_{p}^{1}(\Omega)$,
such that $v=0$ in a neighborhood of the set $F_{0}$, and $v\geq1$
in a neighborhood of the set $F_{1}$, is called the $p$-capacity
of the condenser $E=(F_{0},F_{1};\Omega)$. 

For finite values of $p$-capacity $0<\cp(F_{0},F_{1};\Omega)<+\infty$ and $1<p<\infty$ there exists a unique function $u_{0}$ (an extremal function) such
that: \[
\cp_p(F_{0},F_{1};\Omega)=\int\limits _{\Omega}|\nabla u_{0}|^{p}~dx.\]
 An extremal function is continuous in $\Omega$, monotone in the
domain $\Omega\setminus(F_{0}\cup F_{1})$, equal to zero on $F_{0}$
and is equal to one on $F_{1}$ \cite{HKM,VGR}. 

Extremal functions are dense in Sobolev spaces. Denote by $E_p(\Omega)$ the set of extremal functions for the $p$-capacity of all pairs of connected compact
sets $F_0,F_1\subset \Omega$ with nonempty interior whose boundary points are regular with respect to the open set
$\Omega\setminus(F_0\cup F_1)$.

\begin{thm}
\label{Extr} \cite{VGR} Let $1<p<\infty$. There exists a countable collection of functions $v_k\in E_p(\Omega)$, $k=1,2,...$,
such that, for every function $u\in L^1_p(\Omega)$ and every $\varepsilon>0$, $u$ can be represented as a linear combination 
$u=c_0+\sum\limits_{k=1}^{\infty}c_k v_k$ and
$$
\|u\mid L^1_p(\Omega)\|\leq\sum\limits_{k=1}^{\infty}\|c_kv_k\mid L^1_p(\Omega)\|\leq \|u\mid L^1_p(\Omega)\|+\varepsilon.
$$
\end{thm}

The following $p$-capacitary description of composition operators on Sobolev spaces is correct:

\begin{thm}
\label{CompCap} \cite{VG77} A homeomorphism $\varphi:\Omega\to\widetilde{\Omega}$
induces by the composition rule $\varphi^{\ast}(f)=f\circ\varphi$ a bounded composition
operator 
\[
\varphi^{\ast}:L^1_p(\widetilde{\Omega})\to L^1_p(\Omega),\,\,\,1<p<\infty,
\]
if and only if for any condenser $(F_{0}, F_{1})\subset\widetilde{\Omega}$ the inequality
$$
\cp_p^{\frac{1}{p}}(\varphi^{-1}(F_0), \varphi^{-1}(F_1);\Omega)\leq K_p\cp_p^{\frac{1}{p}}(F_{0}, F_{1};\widetilde{\Omega})
$$
holds.
\end{thm}

This theorem yields that mappings, which generates bounded composition operators on Sobolev spaces, preserve sets of capacity zero. On this base we deal only with quasi-continuous representations of Sobolev functions.

\begin{thm}
\label{CompThP} \cite{GGR} A homeomorphism $\varphi:\Omega\to\widetilde{\Omega}$
between two domains $\Omega$ and $\widetilde{\Omega}$ is a weak $p$-quasiconformal, $1\leq p<\infty$,
if and only if $\varphi$ induces by the composition rule $\varphi^{\ast}(f)=f\circ\varphi$ a bounded composition
operator 
\[
\varphi^{\ast}:L^1_p(\widetilde{\Omega})\to L^1_p(\Omega).
\]
\end{thm}

\subsection{Co-quasiconformal mappings}

In the spectral theory for the Laplace operator the significant role play weak $p$-quasiconformal mappings for $p=n-1$. 

\begin{rem} Note, that in the space $\mathbb R^3$ weak $2$-quasiconformal mappings can be characterized in terms of co-distortion:
$$
|D\varphi(x)|\leq K\min|\adj D\varphi(x)|\,\,\,\text{a.~e. in}\,\,\, \Omega.
$$

On this way it is natural to call weak $(n-1)$-quasiconformal mappings as {\it co-quasiconformal mappings} because these mappings are characterized in terms of "length" and "co-length" (area, if n=3). Similar notions were studied in \cite{S}.
\end{rem}

The class of space co-quasiconformal mappings is a natural analogue of planar quasiconformal mappings from the point of view of possible applications to PDE.
Let us discuss the case $p=n-1$ in more details. For this case we have the following "duality":

\begin{thm}
\label{dual} Let $\varphi:\Omega\to\widetilde{\Omega}$
be a weak $p$-quasiconformal homeomorphism, $p=n-1$. Then the inverse mapping $\varphi^{-1}$ 
induces by the composition rule $(\varphi^{-1})^{\ast}(g)=g\circ\varphi^{-1}$ a bounded composition
operator 
\[
(\varphi^{-1})^{\ast}:L^1_{\infty}({\Omega})\to L^1_{\infty}(\widetilde{\Omega}).
\]
\end{thm}

\begin{proof}

Let $f\in L^1_{\infty}(\widetilde{\Omega})$. Then the composition $f\circ\varphi$ is weakly differentiable in $\Omega$.
Since $\varphi$ is the mapping of finite distortion, $D\varphi(x)=0$ for a.~e. $x\in \Omega$ where $J(x,\varphi)=0$. Hence we may define $\adj D\varphi(x)=0$ at such points.
Then  \cite{GU10}
$$
|J(x,\varphi)||\nabla f|(\varphi(x))\leq |\nabla(f\circ\varphi)|(x)\adj D\varphi(x) \,\,\,\text{for almost all}\,\,\,x\in \Omega.
$$

Because pre-image of a set measure zero has measure zero \cite{VU1}, we obtain
\begin{multline}
\|f\mid L^1_{\infty}(\widetilde{\Omega})\|=\ess\sup\limits_{y\in \widetilde{\Omega}}|\nabla f|(y)=
\ess\sup\limits_{x\in {\Omega}}|\nabla f|(\varphi(x))\\
\leq
\ess\sup\limits_{x\in {\Omega}}
|\nabla (f\circ\varphi)|(\varphi(x))\frac{|\adj D\varphi(x)|}{|J(x,\varphi)|}
\leq 
\ess\sup\limits_{x\in {\Omega}}|\nabla (f\circ\varphi)|(\varphi(x))\frac{|D\varphi(x)|^{n-1}}{|J(x,\varphi)|}\\
\leq K_{n-1}^{n-1}(\Omega)\cdot \|\varphi^{\ast}f\mid L^1_{\infty}(\Omega)\|.
\nonumber
\end{multline}

So, we have the lower estimate for the composition operator
for function $\varphi^{\ast}(f)$ of the class $L^1_{\infty}(\Omega)$.
Therefore, the inverse operator $(\varphi^{\ast})^{-1}=(\varphi^{-1})^{\ast}$ induces by the composition rule a bounded operator 
$$
(\varphi^{-1})^{\ast} : L^1_{\infty}(\Omega)\to L^1_{\infty}(\widetilde{\Omega}).
$$

\end{proof}

\begin{rem} By \$ 3.3 from \cite{M} any bounded composition operator $$
\psi^{\ast} : L^1_{\infty}(\Omega)\to L^1_{\infty}(\widetilde{\Omega}).
$$ is induced by a subareal homeomorphism. Therefore homeomorphisms  which are inverse to weak $(n-1)$-quasiconformal homeomorphisms are subareal mappings.
\end{rem}

Recall, that a homeomorphism $\varphi:\Omega\to \widetilde{\Omega}$ is called subareal \cite{M} if there exists a constant $K<\infty$ such that following inequality holds for any locally Lipschitz $(n-1)$-dimensional manifold $V\subset\Omega$:
$$
S(\varphi(V))\leq K S(V).
$$

\begin{cor}
\label{lip} Let $\varphi:\Omega\to\widetilde{\Omega}$
be a weak $(n-1)$-quasiconformal homeomorphism. Then the inverse mapping $\varphi^{-1}$ belongs to $\operatorname{Lip}(\widetilde{\Omega})$.
\end{cor}

\begin{cor}
\label{oneone} Let $\varphi:\Omega\to\widetilde{\Omega}$
be a weak $(n-1)$-quasiconformal homeomorphism. Then the mapping $\varphi$ 
induces by the composition rule $\varphi^{\ast}(f)=f\circ\varphi$ a bounded composition
operator 
\[
\varphi^{\ast}:L^1_{1}(\widetilde{\Omega})\to L^1_{1}(\Omega).
\]
\end{cor}

\begin{proof}
By Theorem~\ref{CompThP} the mapping $\varphi$ 
induces by the composition rule $\varphi^{\ast}(f)=f\circ\varphi$ a bounded composition
operator 
\[
\varphi^{\ast}:L^1_{1}(\widetilde{\Omega})\to L^1_{1}(\Omega).
\]
if and only if $\varphi$ is a mapping of finite distortion and
$$
K_1(\Omega)=\ess\sup\limits_{x\in\Omega}\frac{|D\varphi(x)|}{|J(x,\varphi)|}<\infty.
$$
By Corollary~\ref{lip}
$$
|D\varphi^{-1}(y)|\leq K_{n-1}^{n-1}(\Omega)\,\,\text{for almost all}\,\, y\in\widetilde{\Omega}.
$$ 
Hence
\begin{multline*}
\frac{|D\varphi(x)|}{|J(x,\varphi)|}
\leq \frac{|D\varphi^{-1}(y)|^{n-1}}{|J(y,\varphi^{-1})|}\frac{1}{|J(x,\varphi)|}\\
= |D\varphi^{-1}(y)|^{n-1}<\infty,
\,\,\,\text{for almost all}\,\,\, y=\varphi(x)\in\widetilde{\Omega}
\end{multline*}
Because $\varphi$ possesses $N^{-1}$-Luzin property, we have that 
$$
\ess\sup\limits_{x\in\Omega}\frac{|D\varphi(x)|}{|J(x,\varphi)|}<\infty
$$
and 
the mapping $\varphi$ 
induces by the composition rule $\varphi^{\ast}(f)=f\circ\varphi$ a bounded composition
operator 
\[
\varphi^{\ast}:L^1_{1}(\widetilde{\Omega})\to L^1_{1}(\Omega).
\]
\end{proof}

Using Marcinkiewicz interpolation theorem \cite{SW} and Corollary~\ref{oneone} we obtain

\begin{thm}
Let $\varphi:\Omega\to\widetilde{\Omega}$
be a $p$-quasiconformal homeomorphism, $p=n-1$. Then the mapping $\varphi$ 
induces by the composition rule $\varphi^{\ast}(f)=f\circ\varphi$ a bounded composition
operator 
\[
\varphi^{\ast}:L^1_{q}(\widetilde{\Omega})\to L^1_{q}(\Omega)
\]
for any $q\in [1,n-1]$.
\end{thm}

\medskip

Note, that another approach to the theory of mappings of finite distortion is based on the notion of modules 
( see, for example, \cite{GRSY, MRSY, S16})

\medskip

{\bf Acknowledgment.} 
The financial support by United States-Israel  Binational Science foundation (Grant 2014055) is gratefully acknowledged.

\end{document}